\documentclass[a4paper]{amsart}
\usepackage{amsmath,amssymb,amsthm}

\newcommand{\C}{{\mathbb C}}

\newcommand{\R}{{\mathbb R}}

\newcommand{\abs}[2][\empty]{\ifx#1\empty\left|#2\right|%
\else#1\vert #2 #1\vert\fi}%optional arg=size
\newcommand{\Cnt}[1][]{{\mathcal C}^{#1}}
\def\conv{\star}%convolution
\newcommand{\eps}{\varepsilon}
\def\restr#1#2{{#1}_{|#2}}%restriction
\newcommand{\supp}{\mathop{\mathrm{supp}}}%support
\newcommand{\test}{\mathcal D}

\newcommand{\ster}[1]{{{}^* \mskip-1mu #1}}
\newcommand{\st}{\mathop\mathrm{st}}
\def\Fin{\mathrm{Fin}}

\newtheorem{thm}{Theorem}[section]
\newtheorem{lemma}[thm]{Lemma}

\theoremstyle{definition}
\newtheorem*{df}{Definition}

\begin{document}
\title{Characterization of distributions having a value at a point in the sense of Robinson}
\author{Hans Vernaeve}
\address{Ghent University, Krijgslaan 281\\
B-9000 Gent, Belgium.}
\email{hvernaev@cage.ugent.be}
\author{Jasson Vindas}
\thanks{J.~Vindas gratefully acknowledges support by a Postdoctoral Fellowship of the Research
Foundation--Flanders (FWO, Belgium)}
\address{Ghent University, Krijgslaan 281\\
B-9000 Gent, Belgium.}
\email{jvindas@cage.ugent.be}

\begin{abstract}
We characterize Schwartz distributions having a value at a single point in the sense introduced by means of nonstandard analysis by A.\ Robinson. They appear to be distributions continuous in a neighborhood of the point.
\end{abstract}

\maketitle

\section{Introduction}
In \cite[\S 5.3]{Rob}, A.\ Robinson initiated the use of nonstandard analysis in the theory of Schwartz distributions. Among other things, he introduces nonstandard representatives of a Schwartz distribution and, by means of a property of the representatives, he introduces a notion of point value of a distribution.

A natural question to raise is how Robinson's notion of point value is related to the classical definition of point value in the sense of \L ojasiewicz \cite{Loj,Loj2}. Through investigation of this question, we arrived at a characterization which is the main result of this paper: a distribution has a value at $x_0\in\Omega$ in the sense of Robinson iff it is a continuous function in a neighborhood of $x_0$. Our characterization improves an earlier result of P. Loeb \cite{Loeb} which has to assume the everywhere existence of Robinson point values (cf. section \ref{Known results}).

We mention that the concept of point value in the sense of \L ojasiewicz has shown to be of great importance
in several areas of mathematical analysis and its applications, such as spectral expansions \cite{EstKan,Walter1,vind-est}, sampling theorems and summability of cardinal series \cite{Walter2}, edge detection from spectral data \cite{est-vind}, or the convergence of wavelet expansions \cite{Pilipovic-Teofanov,Walter3}.

\section{Notations}
By $\Omega$, we always denote an open subset of $\R^n$. We denote $B(a,r):=\{x\in \R^n: \abs{x-a}<r\}$.

\subsection{Schwartz distributions}
We denote by $\test(\Omega)$ the space of functions in $\Cnt[\infty](\Omega)$ with compact support contained in $\Omega$. Its dual space $\test'(\Omega)$ is the space of Schwartz distributions on $\Omega$. We denote the action of a linear map $T$: $\test(\Omega)\to\C$ on an element $\phi\in\test(\Omega)$ by means of the pairing $\left<T,\phi\right>$. Sometimes it is useful to denote functions and distributions by means of its action on a dummy variable (e.g., as in \cite{EstKan}); we then denote the action of $T$ on $\phi$ as $\left<T(x), \phi(x)\right>$. This allows us to write changes of variables $y=F(x)$ simply as $\left<T(F(x)), \phi(x)\right> = \left<T(y), \phi(F^{-1}(y)) \cdot \abs{D F^{-1}(y)}\right>$. We refer to \cite{EstKan,Schw} for further information about Schwartz distributions.

\subsection{Nonstandard analysis}
We work in the framework of nonstandard analysis as introduced by Robinson \cite{Rob}. We refer to \cite{Gol} for a more accessible introduction to nonstandard analysis. As usual, if $f$: $\R^n\to\R^m$ is a function, we keep the notation $f$ for its canonical extension $\ster f$ (and similarly for relations on $\R^n$). Also $\ster\int$ will be denoted by $\int$. We denote the set of all finite numbers in $\ster\C$ by $\Fin(\ster\C)$ and write $x\approx y$ if $\abs{x-y}$ is infinitesimal ($x,y\in\ster \R^n$). We write $x\lessapprox y$ if $x\le y$ or $x\approx y$ ($x,y\in\ster\R$). We denote the standard part (a.k.a.\ shadow) by $\st$.

\section{Known results}\label{Known results}
Robinson works with real valued distributions on the real line, but the generalization to complex valued distributions on an open subset $\Omega$ of $\R^n$ is in most cases straightforward.
We say that a function $f\in\ster{\Cnt[\infty]}(\Omega)$ represents (in the sense of Robinson) a (not necessarily continuous) linear map $T$: $\test(\Omega)\to\C$ if $\int_\Omega f\phi\approx \langle T,\phi \rangle$ for each $\phi\in\test(\Omega)$. In fact, more general functions than $\ster{\Cnt[\infty]}(\Omega)$-functions are allowed as representatives, but $\ster{\Cnt[\infty]}(\Omega)$-functions suffice to develop distribution theory by infinitesimal means. Robinson calls equivalence classes of functions representing the same map $T$ predistributions. We will identify the predistribution with the map $T$.

Robinson calls a predistribution standard at $x_0\in\Omega$ if it has a representative $f$ that is S-continuous at $x_0$, i.e., such that $f(x)\approx f(x_0)$ for each $x\approx x_0$ \cite[p.\ 140]{Rob}.
He shows (with a slightly different proof):
\begin{thm}[Robinson]\label{pointvals}
Let $T$ be a linear map $\test(\Omega)\to\C$. If $T$ has a representative $f$ that is S-continuous at $x_0\in\Omega$, then $f(x_0)\in\Fin(\ster\C)$.
Moreover, the value $\st f(x_0)$ does not depend on the chosen S-continuous representative.
\end{thm}
\begin{proof}
Let $\eps\in\R$, $\eps>0$. Since $f$ is internal and S-continuous, we find by overspill (see e.g.\ \cite[11.9.1]{Gol}) on the set
\[\{r\in\ster\R, r >0: (\forall x\in\ster\Omega)\abs{x-x_0}\le r \implies \abs{f(x)-f(x_0)}\le \eps\}\]
that there exists $r\in\R$, $r>0$ such that $\abs{f(x)-f(x_0)}\le \eps$ for each $x\in\ster B(x_0,r)\subseteq \ster\Omega$.
Now let $\phi\in\test(B(x_0,r))$ with $\int_{\Omega}\phi=1$. Let $C:=\int_{\Omega}\abs\phi\in\R$. Then, since $f$ represents $T$,
\begin{align*}
\abs{\left<T,\phi\right>-f(x_0)}\approx \abs{\int_{\ster\Omega} f(x)\phi(x)\,dx-f(x_0)}
&=\abs{\int_{\ster B(x_0,r)}(f(x)-f(x_0))\phi(x)\,dx}\\
&\le\int_{\ster B(x_0,r)}\eps\abs{\phi(x)}\,dx=C\eps.
\end{align*}
In particular, $f(x_0)\in\Fin(\ster\C)$. For any $g$ representing $T$ and S-continuous at $x_0$, we have the same
inequality (possibly only for some smaller $r\in\R$, $r>0$), so $\abs{f(x_0)-g(x_0)}\lessapprox 2C\eps$.
As $\eps$ is arbitrary, $\st f(x_0)=\st g(x_0)$.
\end{proof}
The number $\st f(x_0)$ is called the value (in the sense of Robinson) of the predistribution at $x_0$. P.\ Loeb \cite{Loeb} proves that if $T$ admits a value $g(x)$ at each $x\in\Omega$, then the resulting map $g$: $\Omega\to\C$ is continuous. In that case, $\ster g$ represents $T$ \cite[5.3.15]{Rob}, and hence $T$ is a continuous function (as a regular Schwartz distribution). Actually, Loeb's result is a particular case of theorem \ref{main}, shown below. 

A distribution $T\in\test'(\Omega)$ admits the value $c\in\C$ at $x_0\in \Omega$ in the sense of \L ojasiewicz \cite{Loj,Loj2} if $\lim_{\eps\to 0}T(x_0+\eps x) = c$, where the limit is interpreted in the distributional sense, i.e., if
\[\lim_{\eps\to 0}\left<T(x_0+\eps x),\phi(x)\right> = \left<c,\phi\right>=c\int_{\Omega}\phi, \quad \forall \phi\in\test(\Omega).\]

Observe that if $T$ is continuous in a neighborhood of $x_0$, then one readily verifies that $T$ has a \L ojasiewicz value at $x_{0}$ and its value is in fact $T(x_{0})=c$. On the other hand, the converse is not true, in general, as shown by the function $T(x)=\left|x\right|^{-\frac{1}{2}}e^{i/x}$, which is unbounded at the origin but it admits the \L ojasiewicz value 0 at $x_{0}=0$. More generally, any function of the form $\left|x\right|^{-\beta}e^{i/\left|x\right|^{\alpha}}$, where $\alpha,\beta>0$, uniquely determines a distribution that has \L ojasiewicz value 0 at the origin \cite{Loj}. 

Of crucial importance for our result is the following theorem \cite[\S6.2]{Loj2}:
\begin{thm}\label{Lojasiewicz}
Let $T\in\test'(\Omega)$ and $x_0\in\Omega$. If $\lim_{\eps\to 0, \lambda \to 0} T(x_0+\eps x + \lambda)$ exists in the distributional sense, i.e., $\lim_{\eps\to 0, \lambda \to 0} \left<T(x_0+\eps x + \lambda),\phi(x)\right>$ exists $\forall\phi\in\test(\Omega)$, then $T$ is a continuous function in a neighborhood of $x_0$.
\end{thm}

\section{New results}
\begin{lemma}\label{distr}
Let $T$ be a predistribution that is standard at $x_0\in\Omega$. Then $T$ is a Schwartz distribution in an open neighborhood $\omega$ of $x_0$ (i.e., $\restr{T}\omega$ is a continuous map $\test(\omega)\to\C$).
\end{lemma}
\begin{proof}
Let $f$ be a representative of $T$ that is S-continuous. As in the proof of theorem \ref{pointvals}, there exists $r\in\R$, $r>0$, such that $\abs{f(x)-f(x_0)}\le 1$ for each $x\in\ster B(x_0,r)\subseteq\ster\Omega$. Let $\phi\in\test(B(x_0,r))$. Then
\begin{multline*}
\abs{\left<T,\phi\right>}\approx \abs{\int_{\ster B(x_0,r)} f\phi}\le \sup_{B(x_0,r)}\abs{f}\mu(B(x_0,r))\sup_{B(x_0,r)}\abs{\phi}\\
\le \big(\abs{f(x_0)}+1\big)\mu(B(x_0,r))\sup_{B(x_0,r)}\abs{\phi},
\end{multline*}
where $\mu$ denotes the Lebesgue measure. Taking standard parts,
\[\abs{\left<T,\phi\right>}\le \big(\abs{\st f(x_0)}+1\big)\mu(B(x_0,r))\sup_{B(x_0,r)}\abs{\phi}.\]
Hence $\restr{T}{B(x_0,r)}$ is continuous.
\end{proof}

\begin{df}
We call $\psi\in\ster\test(\R^n)$ a strict nonstandard delta function (this name corresponds to the standard notion of a \emph{strict delta net}, see e.g.\ \cite[\S 7]{Ob92}) if
\begin{enumerate}
\item $\int_{\ster\R^n}\psi=1$
\item $\psi(x)=0$, $\forall x\not\approx 0$
\item $\int_{\ster\R^n}\abs{\psi}\in\Fin(\ster\R)$.
\end{enumerate}
\end{df}

\begin{lemma}\label{repres-delta}
Let $\psi$ be a strict nonstandard delta function. Let $\Omega$ be a (standard) neighborhood of $0$ and let $f\in\ster{\Cnt[0]}(\Omega)$ be S-continuous at $0$. Then $\int_{\ster\R^n} \psi f \approx f(0)$.
\end{lemma}
\begin{proof}
As $\psi$ is internal, there exists $r\in\ster\R$, $r\approx 0$ such that $\psi(x)=0$ if $\abs x \ge r$ by overspill. Then 
\[\abs{f(0) - \int_{\ster\R^n}\psi f} = \abs{\int_{\ster\R^n} (f(0) - f(x))\psi(x)\,dx}\le \sup_{x\le r}\abs{f(0)-f(x)}\int_{\ster\R^n}\abs\psi\approx 0.\]
\end{proof}
In particular, each strict nonstandard delta function is a representative of the delta distribution, since $\ster \phi$ is S-continuous for each $\phi\in\test(\R^n)$ (see e.g.\ \cite[\S7.1]{Gol}).

\begin{thm}\label{straight}
Let $T$ be a predistribution that admits the value $c$ (in the sense of Robinson) at $x_0\in\Omega$. Then $\left<\ster T(x),\psi(x-x_0)\right>\approx c$ for each strict delta function $\psi$.
\end{thm}
\begin{proof}
Let $\eps\in\R$, $\eps>0$. In the proof of
theorem~\ref{pointvals}, we showed that there exists $r\in\R$, $r>0$ such that for each $\phi\in\test(B(x_0,r))$ with $\int_\Omega\phi=1$ and $\int_\Omega\abs{\phi}= C$, we have $\abs{\left<T,\phi\right>-c}\le (C+1)\eps$.
By overspill, for each $\phi\in\ster\test(B(x_0,r))$ with $\int_{\ster\Omega}\phi=1$ and $\int_{\ster\Omega}\abs{\phi}= C$, we have $\abs{\left<\ster T,\phi\right>-c}\le (C+1)\eps$. This holds in particular for any $\eps\in\R$, $\eps>0$ and $\phi(x)= \psi(x-x_0)$ with $\psi$ a strict nonstandard delta function. Hence $\left<\ster T(x),\psi(x-x_0)\right>\approx c$ for such $\psi$.
\end{proof}
From our main result, theorem \ref{main}, it will follow that the property in the statement of the previous theorem actually characterizes predistributions admitting a value at $x_0$ in the sense of Robinson.

The previous theorem sheds a light on the relation between point values in the sense of Robinson and point values in the sense of \L ojasiewicz. By a nonstandard characterization of limits (see e.g.\ \cite[\S7.3]{Gol}) a distribution $T\in\test'(\Omega)$ admits the value $c$ at $x_0\in\Omega$ in the sense of \L ojasiewicz iff $\left<\ster T(x),\psi(x-x_0)\right>\approx c$ for each so-called model delta function $\psi$, i.e., $\psi(x) = \rho^{-n}\phi(x/\rho)$ with $\phi\in\test(\R^n)$, $\int_{\R^n}\phi = 1$, $\rho\in\ster\R\setminus\{0\}$, $\rho\approx 0$ (the name model delta function corresponds to the standard notion of a \emph{model delta net}, see e.g.\ \cite[\S 7]{Ob92}).

\begin{lemma}\label{local-point}
Let $\omega$ be an open subset of $\Omega$ and let $x_0\in\omega$. Then a linear map $T$: $\test(\Omega)\to\C$ admits the value $c$ at $x_0$ (in the sense of Robinson) iff $\restr{T}{\omega}$: $\test(\omega)\to\C$ admits the value $c$ at $x_0$ (in the sense of Robinson).
\end{lemma}
\begin{proof}
$\Rightarrow$: immediate.\\
$\Leftarrow$: Let $f\in\ster{\Cnt[\infty]}(\omega)$ be a representative of $\restr{T}{\omega}$, i.e., $\int_{\ster\omega}f \phi\approx \left<T,\phi\right>$, $\forall \phi\in\test(\omega)$, and suppose that $f$ is S-continuous in $x_0$. Let $g\in\ster{\Cnt[\infty]}(\Omega)$ be any representative of $T$. Then $f-\restr{g}{\omega}$ is a representative of the $0$-distribution on $\omega$. Let $\chi\in\test(\omega)$ with $\chi=1$ on some (standard) neighborhood $V$ of $x_0$. Then $\int_{\ster\Omega}(f-g)\chi\phi\approx 0$, $\forall \phi\in\test(\Omega)$, since $\chi\phi\in\test(\omega)$. So $(f-g)\chi$ is a representative of $0$ on $\Omega$, and $f\chi+g(1-\chi)$ is a representative of $T$ on $\Omega$ which is equal to $f$ in a (standard) neighborhood of $x_0$.
\end{proof}

\begin{thm}\label{main}
Let $T$ be a predistribution and $x_0\in\Omega$. Then $T$ is standard at $x_0$ iff $T$ is a continuous function in a neighborhood of $x_0$.
\end{thm}
\begin{proof}
$\Rightarrow$: by lemma \ref{distr}, there exists an open neighborhood $\omega$ of $x_0$ such that $T\in\test'(\omega)$. Let $c$ denote the value of $T$ at $x_0$ (in the sense of Robinson). Let $\phi\in\test(\omega)$ with $\int_{\omega}\phi=1$. Let $\eps\in\ster\R$, $\eps\approx 0$ and $\lambda\in\ster\R^n$, $\lambda\approx 0$. Then, by theorem \ref{straight},
\[
\big< \ster T(x_0+\eps x + \lambda),\phi(x)\big> = \Big< \ster T(x), \eps^{-n}\phi\Big(\frac{x-x_0-\lambda}\eps\Big)\Big>\approx c,
\]
since $\psi(x):=\eps^{-n}\phi(\frac{x-\lambda}\eps)\in\ster\test(\R^n)$ is a strict nonstandard delta function:\\
1. by a change of variables, $\int_{\ster\R^n}\psi = \int_{\ster\R^n}\phi=1$.\\
2. $\psi(x)\ne 0$ iff $\frac{x-\lambda}\eps\in \supp\phi$ iff $x\in\lambda+ \eps\supp\phi$; then in particular $x\approx 0$.\\
3. $\int_{\ster\R^n} \abs\psi = \int_{\ster\R^n} \abs\phi\in\R$ is finite.\\
As $\eps$ and $\lambda$ are arbitrary, by a nonstandard characterization of limits (see e.g.\ \cite[\S7.3]{Gol})
\[\lim_{\lambda\to 0, \eps\to 0} \big<T(x_0+\eps x+\lambda), \phi(x)\big> = c\]
for each $\phi\in\test(\omega)$ with $\int_{\omega}\phi=1$. By theorem \ref{Lojasiewicz}, $T$ is a continuous function in a neighborhood of $x_0$.

$\Leftarrow$: let $\omega$ be an open neighborhood of $x_0$ and $f\in\Cnt[0](\omega)$ such that $T=f$ on $\omega$. Let $g:= f\conv\psi\in \ster{\Cnt[\infty]}(\omega')$ for some neighborhood $\omega'$ of $x_0$ ($\conv$ denoting convolution), where $\psi$ is a strict nonstandard delta function. By lemma \ref{repres-delta}, $g(x) = \int_{\ster\R^n}f(x-y)\psi(y)\,dy\approx f(x)$ for each $x\in\ster\omega'$ (lemma \ref{repres-delta} can be applied since $f$ is S-continuous at any $x\in \ster\omega'$ if the closure of $\omega'$ is contained in $\omega$, see e.g.\ \cite[\S7.1]{Gol}). In particular, $g$ is S-continuous at $x_0$. Then for any $\phi\in\test(\omega')$,
\[\abs{\int_{\ster\omega'} g\phi - \left<T,\phi\right>} = \abs{\int_{\ster\omega'} g\phi - \int_{\ster\omega'} f\phi}\le\sup_{\ster\omega'}\abs{g-f} \int_{\omega'}\abs\phi\approx 0,\]
so $g$ represents $T$ on $\omega'$. By lemma \ref{local-point}, $T$ admits the value $\st g(x_0)$ at $x_0$.
\end{proof}

\end{document}